\documentclass[12pt]{article}
\usepackage{amssymb,amsmath,latexsym}

\usepackage{amsthm}
\newtheorem{theorem}{Theorem}
\newtheorem{definition}{Definition}
\newtheorem{remark}{Remark}

\begin{document}
\title{Characterizations of the $(h,k,\mu,\nu)-$Trichotomy for Linear Time-Varying Systems}
\author{Ioan-Lucian~Popa, Traian~Ceau\c su, Mihail~Megan}
\date{}

\maketitle

\abstract{The present paper considers a concept of $(h,k,\mu,\nu)-$trichotomy for noninvertible linear time-varying systems in Hilbert spaces. This work provides a characterization for linear time-varying systems that admits a $(h,k,\mu,\nu)-$trichotomy in terms of two coupled systems having a $(h,\mu,\nu)-$dichotomy.
}

\section{Introduction}
In the qualitative theory of difference equations the notions of dichotomy and trichotomy play a vital role. This fact is very well analysed in \cite{alonso} and \cite{hong} where it is proved that the property of exponential trichotomy is necessary in the presence of ergodic solutions of linear differential and difference equations with ergodic perturbations. Besides the classical concepts  of uniform and nonuniform exponential trichotomy (see \cite{elaydi-jang} and reference therein for more details), in \cite{fener}  J. Lopez-Fener and M. Pinto extend the previous concept to the so called  $(h,k)-$trichotomy using  two sequences with positive terms. Recently, a new nonuniform concept of $(\mu,\nu)-$dichotomy is proposed in \cite{bento} with  increasing sequences which go to infinity. These sequences, called growth rates, have been used to extend the classical concepts of exponential and polynomial dichotomy and trichotomy, both for uniform and nonuniform approaches. We enhance the work developed in \cite{babutia}, \cite{chang}, \cite{popa3}, \cite{popa2}, \cite{popa} and \cite{Zhang} and all included references for a detailed discussion considering this approach.

One problem that lies within the main interest in the asymptotic behavior  of the solution of difference equation is the characterization of the trichotomy in terms of dichotomies. In \cite{elaydi-jang} S. Elaydi and K. Janglajew considered (invertible) difference equations on the entire axis $\mathbb{Z}$ and they proved that if the difference equation has an (E-H) trichotomy that also has exponential dichotomy on $\mathbb{Z}_-$ and $\mathbb{Z}_+.$ This result was extended by S. Matucci \cite{matuci} to the concept of $l^p$ trichotomy. Similarly, in \cite{elaydi-jang} is proved that $l^p$ trichotomy on $\mathbb{Z}$ implies $l^p$ dichotomy on $\mathbb{Z}_+$ and on $\mathbb{Z}_-$ and $l^p$ dichotomy on $\mathbb{Z}$ implies a trivial $l^p$ trichotomy. In \cite{hong2} is proved that for almost periodic difference equations the notion of exponential trichotomy on $\mathbb{Z}$ implies exponential dichotomy on $\mathbb{Z}.$   Tacking into consideration  the previous three results, there naturally arises the question about what happens when we consider (noninvertible) difference equations only defined on $\mathbb{Z}_+.$ This represents the main goal of the present paper, to provide a characterizations for linear time varying systems (i.e. difference equations) that are defined only on $\mathbb{Z}_+$ and not assumed to be invertible. Thus, we consider an generalization from $(h,k,\mu,\nu)-$dichotomy to $(h,k,\mu,\nu)-$trichotomy in discrete time  from \cite{Zhang} and we prove that linear time-varying (LTV) systems that admit such a trichotomy can be characterized in terms of  two coupled systems that admit a $(h,\mu,\nu)-$dichotomy in the sense of J. Lopez-Fener and M. Pinto \cite{fener}, i.e. a concept that uses only sequences with positive terms.
\section{Preliminaries}
In this section we introduce our notations and present some definitions and useful details.
We denoted by $\mathbb{Z}$ the set of real integers, $\mathbb{Z}_+$ is the set of all $n\in\mathbb{Z},$ $n\geq 0,$ and $\mathbb{Z}_-$ is the set of all $n\in\mathbb{Z},$ $n\leq 0.$ We also denote by $\Delta$ the set of all pairs of real integers, $(m,n)$
with $m\geq n\geq 0.$ The norm on the Hilbert space $\mathcal{H}$ and on $\mathcal{B}(\mathcal{H})$ the Banach algebra of $\mathcal{H}$ will be denoted by $\|\cdot\|.$ The identity operator on $\mathcal{H}$ is denoted by $I.$

We will be considering  the discrete-time linear time-varying   system
\begin{equation*}\label{A}\tag{$\mathfrak{A}$}
x_{n+1}=A_{n}x_{n},\;\quad n\in\mathbb{Z}_+
\end{equation*}
where $(A_n)_{n\in\mathbb{Z}_+}\subset\mathcal{B}(\mathcal{H})$ is a given sequence. If for every $n\in\mathbb{Z}_+$ the sequence $(A_n)_{n\in\mathbb{Z}_+}$ is invertible, then the LTV system (\ref{A}) is called {\it reversible}.
The {\it state transition matrix} for the LTV system (\ref{A}) is defined as
\begin{equation}\label{eqAmn}
\mathcal{A}_m^n :=\left\{
\begin{array}{l l}
A_{m-1}\cdots  A_n &if\;\; m > n\\
I &if\;\; m=n.\\
\end{array}\right.
 \end{equation}
 It is obvious that the transition matrix satisfies the propagator property, i.e.
 $\mathcal{A}_m^n\mathcal{A}_n^p=\mathcal{A}_m^p,\;\;\text{for all}\;\;(m,n),(n,p)\in\Delta$
 and every solution of the LTV system (\ref{A}) satisfies
 $x_m=\mathcal{A}_m^nx_n,\;\;\text{for all}\;\;(m,n)\in\Delta.$

A sequence $(P_n)_{n\in\mathbb{Z}_+}\subset\mathcal{B}(\mathcal{H})$ is called a {\it projections sequence} if
$(P_n)^2=P_n,\;\;\text{for all}\;\;n\in\mathbb{Z}_+.$
A projections sequence $(P_n)_{n\in\mathbb{Z}_+}$ with the property
$A_{n}P_n=P_{n+1}A_n\;\;\text{for every}\;\;n\in\mathbb{Z}_+$
is called {\it invariant} for the system (\ref{A}). As a consequence of the invariance property we get the following relation that
$\mathcal{A}_m^nP_n=P_m\mathcal{A}_m^n,$ for all $(m,n)\in\Delta.$

An increasing sequence $\mu:\mathbb{Z}_+\to [1,\infty)$ is a {\it growth rate sequence} if $\mu(0)=1$ and $\lim\limits_{n\to+\infty}\mu(n)=+\infty.$
\begin{definition}\label{D: h-k}
The LTV system (\ref{A}) admits a \it $(h,k,\mu,\nu)-$trichotomy if there exist invariant projections $(P_n^i)_{n\in\mathbb{Z}_+}\subset\mathcal{B}(\mathcal{H})$, $i\in\{1,2,3\}$ satisfying
\begin{equation}\label{orthogonal}
P_n^1+P_n^2+P_n^3=I,\;\;\;P_n^iP_n^j=0
\end{equation}
for all $n\in\mathbb{Z}_{+}$ and $i,j\in\{1,2,3\},$ $i\neq j$ and  there exist growth rate sequences $h_n,$ $k_n,$ $\mu_n,$ $\nu_n$ and some positive constants $K>0,$ $a>0,$ $b\geq 0$ and $\varepsilon\geq 0$ such that
\begin{equation}\label{ht1}
\|\mathcal{A}_m^nP_n^1x\|\leq K\left(\frac{h_n}{h_m}\right)^a\mu_n^\varepsilon\|P_n^1x\|,
\end{equation}
\begin{equation}\label{ht2}
\|P_n^2x\|\leq K\left(\frac{k_n}{k_m}\right)^b\nu_m^\varepsilon\|\mathcal{A}_m^nP_n^2x\|,
\end{equation}
\begin{equation}\label{ht3}
\|\mathcal{A}_m^nP_n^3x\|\leq K\left(\frac{h_m}{h_n}\right)^a\mu_n^\varepsilon\|P_n^3x\|,
\end{equation}
\begin{equation}\label{ht4}
\|P_n^3x\|\leq K\left(\frac{k_m}{k_n}\right)^b\nu_m^\varepsilon\|\mathcal{A}_m^nP_n^3x\|,
\end{equation}
for all $(m,n)\in\Delta,$ $x\in\mathcal{H}$ and
the restriction of ${\mathcal{A}_m^n}_{|Ker P_n^i}:Ker P_n^i\to Ker P_m^i$ is an isomorphism for all $(m,n)\in\Delta$ and all $i\in\{2,3\}.$ Moreover, the LTV system (\ref{A}) admits a $(h,k,\mu,\nu)-$dichotomy if admits a $(h,k,\mu,\nu)$-trichotomy with $P_n^3=0,$ for all $n\in\mathbb{Z}_+.$
\end{definition}
The notion of $(h,k,\mu,\nu)-$trichotomy is a natural generalization of the classical concepts of uniform (nonuniform) concepts of exponential and polynomial trichotomy (\cite{li}, \cite{song}). The constants a and b play the role of Lyapunov exponents while $\varepsilon$
measures the nonuniformity of trichotomies. In \cite{Zhang} is pointed out that $a\varepsilon<0$ can simplify the previous expressions (i.e. (\ref{ht1})-(\ref{ht4})), but in this way we can see better how the Lyapunov exponents are involved in this characterization.
One can see that not all the definitions the sequences considered are assumed to be growth rate sequences. For example in \cite{fener}, J. L\'{o}pez-Fenner and M. Pinto simply consider sequences with positive terms. In the same line of reasoning, if we consider in previous definition  $h_n,$ $k_n,$ $\mu_n,$ $\nu_n$ four sequences with positive terms we can denote this notion as FP $(h,k,\mu,\nu)$-trichotomy and for $P_n^3=0$ the notion of FP $(h,k,\mu,\nu)$-dichotomy. Also we can mention a particular case of previous notion that will be used in this paper, i.e. FP $(h,\mu,\nu)-$dichotomy obtained for $h_n=k_n.$
\begin{remark}\label{R: 2-4}
(See, \cite{popa}) One can see that relation (\ref{orthogonal}), the {\it orthogonality property} from Definition \ref{D: h-k} can be rewritten as follows:
\begin{description}
\item[a)] in terms of two projection sequences, i.e. there exists $(Q_n^i)_{n\in\mathbb{Z}_+}\subset\mathcal{B}(\mathcal{H})$, $i\in\{1,2\}$ such that
\begin{equation}\label{ec 2 proiectoir}
Q_n^1Q_n^2=Q_n^2Q_n^1=0,\;\;\text{for all}\;\;n\in\mathbb{Z}_{+}
\end{equation}
where $Q_n^1=P_n^1$ and $Q_n^2=P_n^2+P_n^3.$

\item[b)] in terms of four projection sequences, i.e. there exists four projection sequences $(R_n^i)_{n\in\mathbb{Z}_+}\subset\mathcal{B}(\mathcal{H})$, $i\in\{1,2,3,4\}$ such that
\begin{align}\label{ec 4 proiectori}
&R_n^1+R_n^4=R_n^2+R_n^3=I,\;\;\;R_n^1R_n^2=R_n^2R_n^1=0,\nonumber\\
&R_n^3R_n^4=R_n^4R_n^3,\;\;\text{for all}\;\;n\in\mathbb{Z}_{+}
\end{align}
where $R_n^1=P_n^1,$ $R_n^2=P_n^2,$ $R_n^3=P_n^1+P_n^3$ and $R_n^4=P_n^2+P_n^3.$
\end{description}
\end{remark}
\begin{remark}\label{R: 2}
\begin{description}
\item[a)] If we consider $P_n^1=Q_n^1,$ $P_n^2=Q_n^2,$ $P_n^3=I-Q_n^1-Q_n^2,$ then $(P_n^i),$ $i\in\{1,2,3\}$ are orthogonal projection sequences if and only if $(Q_n^i),$ $i\in\{1,2\}$ are orthogonal projection sequences.

\item[b)] If $(P_n^i),$ $i\in\{1,2,3\}$ are orthogonal projection sequences then $(R_n^i),$ $i\in\{1,2,3,4\}$ defined by $R_n^1=P_n^1,$ $R_n^2=P_n^2,$ $R_n^3=P_n^1+P_n^3$ and $R_n^4=P_n^2+P_n^3$ are also orthogonal. Conversely, we have that $P_n^1=R_n^1,$ $P_n^2=R_n^2$ and $P_n^3=R_n^3R_n^4,$ $n\in\mathbb{Z}_+.$
\end{description}
\end{remark}

\section{Main Results}

We associate to the LTV system (\ref{A}) the system
\begin{equation*}\label{B}\tag{$\mathfrak{B}$}
x_{n+1}=B_{n}x_{n},\;\quad n\in\mathbb{Z}_+
\end{equation*}
with $B_n=\left(\frac{h_{n+1}}{h_n}\right)^{a/2}\left(\frac{k_{n+1}}{k_n}\right)^{b/2}A_n.$ Following (\ref{eqAmn}) we have that the state transition matrix associated to (\ref{B}) satisfies
\begin{equation}\label{eq B}
\mathcal{B}_m^n=\left(\frac{h_m}{h_n}\right)^{a/2}\left(\frac{k_m}{k_n}\right)^{b/2}\mathcal{A}_m^n,\;\;\text{for all}\;\;(m,n)\in\Delta.
\end{equation}
In addition, using the orthogonality property from Definition \ref{D: h-k}, we have that $P_n^1$ and $P_n^2+P_n^3$ are also projections sequences which are invariant for the LTV system (\ref{B}), i.e.
\begin{equation*}
\mathcal{B}_m^nP_n^1=P_n^1\mathcal{B}_m^n
\end{equation*}
and
\begin{equation*}
\mathcal{B}_m^n(P_n^2+P_n^3)=(P_m^2+P_m^3)\mathcal{B}_m^n
\end{equation*}
for all $(m,n)\in\Delta.$

\begin{theorem}\label{T: sysB}
If the LTV system (\ref{A}) is $(h,k,\mu,\nu)-$trichotomic then the LTV system (\ref{B}) is FP $(\tilde{h},\mu,\nu)-$dichotomic, where $\tilde{h}_n=\dfrac{h_n^a}{k_n^b}.$
\end{theorem}
\begin{proof}
Take $P_n^1=S_n^1.$ By (\ref{ht1}) we have
\begin{equation*}
\|\mathcal{A}_m^nS_n^1x\|\leq K\left(\frac{h_n}{h_m}\right)^a\mu_n^{\varepsilon}\|S_n^1x\|,
\end{equation*}
for all $(m,n)\in\Delta$ and $x\in\mathcal{H}.$ Using Remark \ref{R: 2} $a)$ if follows that the orthogonality property allow us to consider the following Pythagoras equality
\begin{equation*}
\|(P_n^2+P_n^3)x\|^2=\|P_n^2x\|^2+\|P_n^3x\|^2,
\end{equation*}
for all $n\in\mathbb{Z}_+$ and $x\in\mathcal{H}.$ Thus, we obtain
\begin{align*}
\|(P_n^2+P_n^3)x\|^2&\leq K^2\left(\frac{k_n}{k_m}\right)^{2b}\nu_m^{2\varepsilon}\|P_m^2\mathcal{A}_m^nx\|^2+K^2\left(\frac{k_m}{k_n}\right)^{2b}\nu_m^{2\varepsilon}\|P_m^3\mathcal{A}_m^nx\|^2\\
&\leq K^2\left(\frac{k_m}{k_n}\right)^{2b}\nu_m^{2\varepsilon}\|\mathcal{A}_m^n(P_n^2+P_n^3)x\|^2
\end{align*}
Hence, for $P_n^2+P_n^3=S_n^2$ we have that
\begin{equation*}
\|S_n^2x\|\leq K\left(\frac{k_m}{k_n}\right)^{b}\nu_m^\varepsilon\|\mathcal{A}_m^nS_n^2x\|,
\end{equation*}
for all $(m,n)\in\Delta$ and all $x\in\mathcal{H}.$  We consider the sequence $\tilde{h}_n:\mathbb{Z}_+\to(0,\infty)$ defined by $\tilde{h}_n=\dfrac{h_n^a}{k_n^b}.$ Taking into account relation (\ref{eq B}) we have that
\begin{equation*}
\|\mathcal{B}_m^nS_n^1x\|\leq K\left(\frac{\tilde{h}_n}{\tilde{h}_m}\right)^{1/2}\mu_n^\varepsilon\|S_n^1x\|
\end{equation*}
\begin{equation*}
\|S_n^2x\|\leq K\left(\frac{\tilde{h}_n}{\tilde{h}_m}\right)^{1/2}\nu_m^\varepsilon\|\mathcal{B}_m^nS_n^2x\|
\end{equation*}
for all $(m,n)\in\Delta$ and $x\in\mathcal{H}.$ Further, one can easily see that ${\mathcal{B}_m^n}_{|Ker S_n^2}:Ker S_n^2\to Ker S_m^2$ is an isomorphism for all $(m,n)\in\Delta$ and so it follows that LTV system (\ref{B}) is FP $(\tilde{h},\mu,\nu)-$dichotomic with projections $P_n^1$ and $P_n^2+P_n^3.$ This completes the proof of the theorem.
\end{proof}
Now we associate to (\ref{A}) the LTV system
\begin{equation*}\label{C}\tag{$\mathfrak{C}$}
x_{n+1}=C_{n}x_{n},\;\quad n\in\mathbb{Z}_{+}
\end{equation*}
with $C_n=\left(\frac{h_n}{h_{n+1}}\right)^{a/2}\left(\frac{k_n}{k_{n+1}}\right)^{b/2}A_n.$ Obvious, the state transition matrix $\mathcal{C}_m^n$ associated to the LTV system (\ref{C}) checks
\begin{equation}\label{eq C}
\mathcal{C}_m^n=\left(\frac{h_n}{h_m}\right)^{a/2}\left(\frac{k_n}{k_m}\right)^{b/2}\mathcal{A}_m^n,\;\;\text{for all}\;\;(m,n)\in\Delta.
\end{equation}
Also,
one can see that $P_n^2$ and $P_n^1+P_n^3$ are projections sequences which are invariant for the system (\ref{C}). Moreover, we have that LTV systems (\ref{B}) and (\ref{C}) are couplet together, i.e.
\begin{equation*}
\mathcal{C}_m^n=\left(\frac{h_n}{h_m}\right)^a\left(\frac{k_n}{k_m}\right)^b\mathcal{B}_m^n, \;\;\text{for all}\;\;(m,n)\in\Delta.
\end{equation*}
\begin{theorem}\label{T: sysC}
If the LTV system (\ref{A}) is $(h,k,\mu,\nu)-$trichotomic then the LTV system (\ref{C}) is FP $(\bar{h},\mu,\nu)-$dichotomic, where $\bar{h}_n=\dfrac{1}{\tilde{h}_n}.$
\end{theorem}
\begin{proof}
Proceeding as in the proof of Theorem \ref{T: sysB}, we consider the projection sequences $T_n^1=P_n^1+P_n^3$ and $T_n^2=P_n^2.$ It follows from Remark \ref{R: 2} $a)$
and Pythagoras equality, i.e.
\begin{equation*}
\|(P_n^1+P_n^3)x\|^2=\|P_n^1x\|^2+\|P_n^3x\|^2
\end{equation*}
that for all $(m,n)\in\Delta$ and $x\in\mathcal{H}$ we have
\begin{equation*}
\|\mathcal{A}_m^nT_n^1x\|\leq K\left(\frac{h_m}{h_n}\right)^a\mu_n^{\varepsilon}\|T_n^1x\|
\end{equation*}
\begin{equation*}
\|T_n^2x\|\leq K\left(\frac{k_n}{k_m}\right)^{b}\nu_m^\varepsilon\|\mathcal{A}_m^nT_n^2x\|.
\end{equation*}
We consider the sequence $\bar{h}_n:\mathbb{Z}_+\to(0,\infty)$ defined by $\bar{h}_n=\dfrac{1}{\tilde{h_n}}.$ From equation (\ref{eq C}) if follows that
\begin{equation*}
\|\mathcal{C}_m^nT_n^1x\|\leq K\left(\frac{\bar{h}_n}{\bar{h}_m}\right)^{1/2}\mu_n^\varepsilon\|T_n^1x\|
\end{equation*}
and
\begin{equation*}
\|T_n^2x\|\leq K\left(\frac{\bar{h}_n}{\bar{h}_m}\right)^{1/2}\nu_m^\varepsilon\|\mathcal{C}_m^nT_n^2x\|
\end{equation*}
for all $(m,n)\in\Delta$ and $x\in\mathcal{H}.$ In addition, ${\mathcal{C}_m^n}_{|Ker T_n^2}:Ker T_n^2\to Ker T_m^2$ is an isomorphism for all $(m,n)\in\Delta.$ This allows us to show that LTV system (\ref{C}) admits a FP $(\bar{h},\mu,\nu)-$dichotomy with projection sequences $P_n^2$ and $P_n^1+P_n^3.$
\end{proof}
\begin{remark}\label{R: 8}
Using $(h,k,\mu,\nu)-$trichotomy we have that there exists the projection sequences $(S_n^i)_{n\in\mathbb{Z}_+}\subset\mathcal{B}(\mathcal{H}),$ $(T_n^i)_{n\in\mathbb{Z}_+}\subset\mathcal{B}(\mathcal{H}),$ $i\in\{1,2\}$ satisfying the following properties
\begin{equation*}
S_n^1+S_n^2=I,\;\;T_n^1+T_n^2=I;\;\;S_n^1S_n^2=S_n^2S_n^1=0,\;\;T_n^1T_n^2=T_n^2T_n^1=0;
\end{equation*}
\begin{equation*}
S_n^1T_n^1=T_n^1S_n^1=S_n^1,\;\;S_n^2T_n^1=T_n^1S_n^2=S_n^2-T_n^2=T_n^1-S_n^1,\;\;T_n^2S_n^2=S_n^2T_n^2=T_n^2,\;\;T_n^2S_n^1=S_n^1T_n^2=0.
\end{equation*}
for all $n\in\mathbb{Z}_+.$
It is obvious that if we consider $Q_n^1=S_n^1$ and $Q_n^2=T_n^2$ then we have that $Q_n^1$ and $Q_n^2$ are orthogonal.
\end{remark}
%
\begin{remark}\label{R: 9}
Using $(h,k,\mu,\nu)-$trichotomy property and Theorems \ref{T: sysB} and \ref{T: sysC} we conclude that system (\ref{A}) is reducible to the coupled systems (\ref{B}) and (\ref{C}) such that
\begin{equation}\label{eq1 COU}
\|\mathcal{B}_m^nS_n^1x\|\leq K\left(\frac{\tilde{h}_n}{\tilde{h}_m}\right)^{1/2}\mu_n^\varepsilon\|S_n^1x\|
\end{equation}
\begin{equation}\label{eq2 COU}
\|S_n^2x\|\leq K\left(\frac{\tilde{h}_n}{\tilde{h}_m}\right)^{1/2}\nu_m^\varepsilon\|\mathcal{B}_m^nS_n^2x\|
\end{equation}
\begin{equation}\label{eq3 COU}
\|\mathcal{C}_m^nT_n^1x\|\leq K\left(\frac{\bar{h}_n}{\bar{h}_m}\right)^{1/2}\mu_n^\varepsilon\|T_n^1x\|
\end{equation}
\begin{equation}\label{eq4 COU}
\|T_n^2x\|\leq K\left(\frac{\bar{h}_n}{\bar{h}_m}\right)^{1/2}\nu_m^\varepsilon\|\mathcal{C}_m^nT_n^2x\|
\end{equation}
and ${\mathcal{B}_m^n}_{|Ker S_n^2}:Ker S_n^2\to Ker S_m^2,$ ${\mathcal{C}_m^n}_{|Ker T_n^2}:Ker T_n^2\to Ker T_m^2$ are isomorphisms for all $(m,n)\in\Delta$ and $x\in\mathcal{H}.$
\end{remark}
\begin{theorem}\label{T: sysAh-trich}
If LTV coupled systems (\ref{B}) and (\ref{C}) admits  FP $(\tilde{h},\mu,\nu)-$dichotomy respectively FP $(\bar{h},\mu,\nu)-$dichotomy then the LTV system (\ref{A})  admits a $(h,k,\mu,\nu)-$trichotomy.
\end{theorem}
\begin{proof}
Take $P_n^1=S_n^1.$ By (\ref{eq1 COU})  we have that
\begin{align*}
\|\mathcal{A}_m^nP_n^1x\|&=\left(\frac{h_n^a}{h_m^a}\right)^{1/2}\left(\frac{k_n^b}{k_m^b}\right)^{1/2}\|\mathcal{B}_m^nS_n^1x\|\\
&\leq \left(\frac{h_n^a}{h_m^a}\right)^{1/2}\left(\frac{k_n^b}{k_m^b}\right)^{1/2} K \left(\frac{h_n^a}{k_n^b}\right)^{1/2}\left(\frac{k_m^b}{h_m^a}\right)^{1/2}\mu_n^\varepsilon\|S_n^1x\|\\
&=K \left(\frac{h_n}{h_m}\right)^a\mu_n^\varepsilon\|P_n^1x\|
\end{align*}
for all $(m,n)\in\Delta$ and $x\in\mathcal{H}.$ From (\ref{eq4 COU}), considering $P_n^2=T_n^2$, we obtain
\begin{align*}
\|P_n^2x\|&\leq K \left(\frac{k_n^b}{h_n^a}\right)^{1/2}\left(\frac{h_m^a}{k_m^b}\right)^{1/2}\nu_m^\varepsilon\|\mathcal{C}_m^nT_n^2x\|\\
&= K \left(\frac{k_n^b}{h_n^a}\right)^{1/2}\left(\frac{h_m^a}{k_m^b}\right)^{1/2}\nu_m^\varepsilon \left(\frac{h_n^a}{h_m^a}\right)^{1/2}\left(\frac{k_n^b}{k_m^b}\right)^{1/2}\|\mathcal{A}_m^nP_n^2x\|\\
&=K \left(\frac{k_n}{k_m}\right)^b\nu_m^\varepsilon\|\mathcal{A}_m^nP_n^2x\|\\
\end{align*}
for all $(m,n)\in\Delta$ and $x\in\mathcal{H}.$ Now, we consider $P_n^3=T_n^1S_n^2.$ By Remark \ref{R: 8}, (\ref{eq2 COU}) and (\ref{eq3 COU}) we obtain
\begin{align*}
\|\mathcal{A}_m^nP_n^3x\|&=\|\mathcal{A}_m^nT_n^1S_n^2x\|=\left(\frac{h_m^a}{h_n^a}\right)^{1/2}\left(\frac{k_m^b}{k_n^b}\right)^{1/2}\|\mathcal{C}_m^nT_n^1S_n^2x\|\\
&\leq \left(\frac{h_m^a}{h_n^a}\right)^{1/2}\left(\frac{k_m^b}{k_n^b}\right)^{1/2} K \left(\frac{k_n^b}{h_n^a}\right)^{1/2}\left(\frac{h_m^a}{k_m^b}\right)^{1/2}\mu_n^\varepsilon\|T_n^1S_n^2x\|\\
&=K \left(\frac{h_m}{h_n}\right)^a\mu_n^\varepsilon\|P_n^3x\|
\end{align*}
and, for $P_n^3=S_n^2T_n^1$, we have
\begin{align*}
\|P_n^3x\|&=\|S_n^2T_n^1x\|\leq K \left(\frac{h_n^a}{k_n^b}\right)^{1/2}\left(\frac{k_m^b}{h_m^a}\right)^{1/2}\nu_m^\varepsilon\|\mathcal{B}_m^nS_n^2T_n^1x\|\\
&= K \left(\frac{h_n^a}{k_n^b}\right)^{1/2}\left(\frac{k_m^b}{h_m^a}\right)^{1/2}\nu_m^\varepsilon \left(\frac{h_m^a}{h_n^a}\right)^{1/2}\left(\frac{k_m^b}{k_n^b}\right)^{1/2}\|\mathcal{A}_m^nP_n^3x\|\\
&=K \left(\frac{k_m}{k_n}\right)^b\nu_m^\varepsilon\|\mathcal{A}_m^nP_n^3x\|\\
\end{align*}
for all $(m,n)\in\Delta$ and $x\in\mathcal{H}.$ Moreover, ${\mathcal{A}_m^n}_{|Ker P_n^i}:Ker P_n^i\to Ker P_m^i$ is an isomorphism for all $(m,n)\in\Delta$ and all $i\in\{2,3\}.$
This implies that LTV system (\ref{A}) is  $(h,k,\mu,\nu)-$trichotomic, which concludes our proof.
\end{proof}
Combining Theorems \ref{T: sysB}, \ref{T: sysC} and \ref{T: sysAh-trich} we have the following
\begin{theorem}

The LTV system (\ref{A}) is $(h,k,\mu,\nu)-$trichotomic with projection sequences $(P_n^i),$ $i\in\{1,2,3\}$ if and only if
\begin{description}
\item[a)] (\ref{B}) if FP $(\tilde{h},\mu,\nu)$-dichotomic with projection sequences $(S_n^i),$ $i\in\{1,2\};$

\item[b)] (\ref{C}) if FP $(\bar{h},\mu,\nu)$-dichotomic with projection sequences $(T_n^i),$ $i\in\{1,2\}.$
\end{description}
\end{theorem}
\begin{remark}
The theorems from this article include and generalize the case considered in \cite{lapadat}. For the continuous case of evolution operators we can refer to \cite{kovacs}.
\end{remark}

\section{Conclusion}

In this paper, we have considered the problem of $(h,k\mu,\nu)-$trichotomy for (noninvertible) linear time-varying systems in Hilbert spaces.
It is proved that noninvertible LTV systems defined on $\mathbb{Z}_{+}$ that admit a $h,k,\mu,\nu)$-trichotomy can be decomposed into two coupled systems having
$(h,\mu,\nu)$-dichotomy.

\begin{flushleft}                                                     
     Ioan-Lucian Popa\\
     Department of Mathematics  \\
     "1 Decembrie 1918" University of Alba Iulia \\
     510009-Alba Iulia, Romania \\
     email: {\em lucian.popa@uab.ro}
\end{flushleft}

\begin{flushleft}                                                     
     Traian Ceau\c su\\
     Department of Mathematics  \\
     West University of Timi\c soara, \\
     300223-Timi\c soara, \\
     email: {\em ceausu@math.uvt.ro}
\end{flushleft}

\begin{flushleft}                                                     
     Mihail Megan\\
     Academy of Romanian Scientists  \\
     050094 Bucharest, Romania \\
     email: {\em megan@math.uvt.ro}
\end{flushleft}
\end{document}